\documentclass[11pt]{article}

\usepackage[T1]{fontenc}
\usepackage[english]{babel}
\usepackage[a4paper,margin=1in]{geometry}
\usepackage{amsmath,amssymb,amsthm,mathtools,bm}
\usepackage{booktabs,array}
\usepackage{microtype}
\usepackage{hyperref}
\usepackage{enumitem}
\usepackage{cleveref}
\allowdisplaybreaks
\hypersetup{
  hidelinks,
  pdftitle={Refined Bounds for Unions of Several Parts in Balanced Graph Partitions},
  pdfauthor={Zhanping Yang},
  pdfsubject={Judicious partitions and unions of several color classes},
  pdfkeywords={judicious partitions, balanced graph partitions, induced edges, probabilistic method, centered quadratic forms, Cantelli inequality}
}

\newtheorem{theorem}{Theorem}[section]
\newtheorem{lemma}[theorem]{Lemma}
\newtheorem{conjecture}[theorem]{Conjecture}
\newtheorem{corollary}[theorem]{Corollary}

\newtheorem{claim}[theorem]{Claim}

\theoremstyle{remark}

\newcommand{\E}{\mathbb{E}}

\newcommand{\Var}{\operatorname{Var}}

\newcommand{\calB}{\mathcal{B}}

\crefname{theorem}{theorem}{theorems}
\crefname{lemma}{lemma}{lemmas}
\crefname{conjecture}{conjecture}{conjectures}
\crefname{corollary}{corollary}{corollaries}
\crefname{proposition}{proposition}{propositions}
\crefname{remark}{remark}{remarks}

\title{Bounds for Unions of Several Parts in Balanced Graph Partitions}
\author{
  Zhanping Yang\thanks{ Center for Discrete Mathematics, Fuzhou University, Fuzhou, Fujian 350108, China. Research supported by National Key R\&D Program of China (Grant No. 2023YFA1010202) and National Natural Science Foundation of China (Grant No. 12371342). Email: \texttt{yangzp@fzu.edu.cn}.}
}
\date{}

\begin{document}
\maketitle

\begin{abstract}
Let $k\ge3$ and $1\le \ell\le k-1$.  We study balanced $k$-partitions of a graph for which the union of any $\ell$ parts induces few edges.  We show that every graph $G$ with $n$ vertices
and $m$ edges admits a balanced partition $V_1,\ldots,V_k$ such that
\begin{equation*}
\max_{\substack{A\in\binom{[k]}{\ell}}}e_G\left(\bigcup_{i\in A}V_i\right)\le\frac{\ell^2}{k^2}m+\frac{\ell^2(k-\ell)}{k^2}(n-1)+\frac{\ell(k-\ell)}{k(k-1)}\sqrt{\left(\binom{k}{\ell}-1\right)m}.
\end{equation*}
In the case $\ell=2$, our result confirms a conjecture of Bollob\'as and Scott in a stronger form.
\end{abstract}

\noindent\textbf{2020 Mathematics Subject Classification.} 05C35, 05D40.

\noindent\textbf{Keywords.} judicious partition, balanced graph partition, union of part, induced edge

\section{Introduction}\label{sec:intro}
For a simple graph $G$, we use $V(G)$ and $E(G)$ to denote the vertex set and edge set of $G$,  respectively. We use $\delta(G)$ to denote the minimum degree of $G$. For a vertex set $S\subseteq V(G)$, let $e_G(S)$ denote the number of edges of $G$ with both endpoints in $S$.  A partition $V(G)=V_1\cup\cdots\cup V_k$ is called \emph{balanced} if $\bigl||V_i|-|V_j|\bigr|\le1$ for all $i,j\in[k]$.

Many classical graph partition problems optimize a single edge quantity. The Max-Cut problem, for example, asks for a bipartition with as many crossing edges as possible. Edwards~\cite{Edwards1973,Edwards1975} proved the essentially best possible bound that every graph with $m$ edges has a bipartite subgraph with at least $m/2+(\sqrt{2m+1/4}-1/2)/4$ edges. Alon~\cite{Alon1996} later showed that, for certain values of $m$, this bound can be improved by an additive term of order $m^{1/4}$. Bollob\'as and Scott~\cite{BollobasScottMaxCut2002} extended the Edwards-type estimate to $k$-partitions, obtaining a $k$-partite subgraph with at least 
\begin{equation*}
\frac{k-1}{k}m+\frac{k-1}{2k}\left(\sqrt{2m+\frac{1}{4}}-\frac{1}{2}\right)-\frac{(k-2)^2}{8k}
\end{equation*}
edges. Alon, Bollob\'as, Krivelevich, and Sudakov~\cite{AlonBollobasKrivelevichSudakov2003} further studied maximum cuts and judicious partitions in graphs without short cycles, giving stronger estimates for such graphs and relating large cuts to good judicious partitions.

A different direction is to optimize several quantities simultaneously. Bollob\'as and Scott~\cite{BollobasScott1993} developed this viewpoint under the name of judicious partitioning. In the hypergraph setting, Bollob\'as and Scott~\cite{BollobasScottHypergraphs1997} proved that every $3$-uniform hypergraph with $m$ edges admits a partition into $k$ parts such that each part spans at most $(1+o(1))m/k^3$ edges, which is asymptotically best possible. For graphs, Bollob\'as and Scott~\cite{BollobasScottExact1999} proved that there exists a $k$-partition
$V_1,\ldots,V_k$ such that each part satisfies
\begin{equation*}
e_G(V_i)\le \frac{m}{k^2}+\frac{k-1}{2k^2}\left(\sqrt{2m+\frac{1}{4}}-\frac{1}{2}\right).
\end{equation*}
Scott~\cite{ScottSurvey2005} gives a broader survey of these and related judicious partition problems.

The simultaneous control of cut edges and internal edges was developed further in a series of papers. Xu and Yu~\cite{XuYu2009} showed that one can obtain the above internal-edge bound while at the same time retaining a large $k$-cut, thereby answering a question of Bollob\'as and Scott~\cite{BollobasScottJudicious2002}. Xu and Yu~\cite{XuYu2011} strengthened this by obtaining a partition satisfying
\begin{equation*}
e_G(V_i)\le \frac{m}{k^2}+\frac{k-1}{2k^2}\left(\sqrt{2m+\frac{1}{4}}-\frac{1}{2}\right)
\end{equation*}
for every $i$ together with
\begin{equation*}
e_G(V_1,\ldots,V_k)\ge \frac{k-1}{k}m+\frac{k-1}{2k}\left(\sqrt{2m+\frac{1}{4}}-\frac{1}{2}\right)-\frac{17k}{8},
\end{equation*}
and also settled another related problem of Bollob\'as and Scott. Fan, Hou, and Zeng~\cite{FanHouZeng2014} obtained a further refinement.

 Lee, Loh, and Sudakov~\cite{LeeLohSudakov2013} extended classical cut bounds to bisections and obtained judicious bisection estimates for graphs of large minimum degree. In particular, if the minimum degree is $2r$ or $2r+1$, their results give a bisection $S_1,S_2$ satisfying
\begin{equation*}
\max\{e_G(S_1),e_G(S_2)\}\le\left(\frac{r+1}{2(2r+1)}+o(1)\right)m.
\end{equation*}
Xu and Yu~\cite{XuYu2014} subsequently confirmed a conjecture of Bollob\'as and Scott~\cite{BollobasScottJudicious2002} by proving that every graph $G$ with $\delta(G)\geq2$ has a bisection satisfying $\max\{e_G(S_1),e_G(S_2)\}\le {m}/{3}$ with the triangle as the unique extremal graph.

We now turn to the problem most closely related to the present paper: controlling the edges induced by unions of several parts. Bollob\'as and Scott~\cite{BollobasScottJudicious2002} asked for the smallest function $f(k,m)$ such that every graph with $m$ edges admits a $k$-partition $V_1,\ldots,V_k$ satisfying $e_G(V_i\cup V_j)\le f(k,m)$ for every $1\le i<j\le k$ and proposed several conjectures concerning this quantity. Ma, Yen, and Yu~\cite{MaYenYu2010} obtained several judicious partition results for graphs and hypergraphs; in particular, their graph results imply the asymptotic estimate $f(3,m)\le {m}/{2}+O(m^{4/5})$. Ma and Yu~\cite{MaYu2010} proved that $f(k,m)<{1.6m}/{k}+o(m)$ with the coefficient improved to $1.5/k$ for $k\ge23$. They further showed that if $\delta(G)\ge \varepsilon n$, then there is a $k$-partition satisfying
\begin{equation*}
\max_{1\le i<j\le k}e_G(V_i\cup V_j)\le\frac{4m}{k^2}+\left(\sqrt{\frac{2}{\varepsilon}}+\sqrt{8\log\binom{k}{2}}\right)m^{5/6}.
\end{equation*}
Thus the natural leading term $4m/k^2$ already appears for dense graphs. Bollob\'as and Scott~\cite{BollobasScottJudicious2002} also conjectured the universal bound
\begin{equation*}
\max_{1\le i<j\le k}e_G(V_i\cup V_j)\le\frac{12m}{(k+1)(k+2)}+O(n).
\end{equation*}
Fan and Hou~\cite{FanHou2017} confirmed this conjecture. More precisely, if $\Delta(G)$ is the maximum degree of $G$, they proved that one can achieve
\begin{equation*}
\max_{1\le i<j\le k}e_G(V_i\cup V_j)\le\frac{4m}{k^2}+\frac{4\Delta(G)}{k}+o(m^{7/8}),
\end{equation*}
and also obtained the asymptotically sharp estimate $f(k,m)\le {m}/{(k-1)}+o(m)$ whose leading term is forced by stars. Jin and Xu~\cite{JinXu2019} extended the study from pairs to unions of several parts in uniform hypergraphs. In particular, when the maximum degree is $o(m)$, their result for graphs gives a $k$-partition for which, for every $\ell$-tuple of distinct parts,
\begin{equation*}
e_G(V_{j_1}\cup\cdots\cup V_{j_\ell})\le\frac{\ell^2}{k^2}m+o(m).
\end{equation*}
This asymptotic leading term is the starting point for the general balanced problem considered here. The following sharper conjecture of Bollob\'as and Scott~\cite{BollobasScottJudicious2002} concerns the pair case and keeps the second-order term explicit.
\begin{conjecture}[\textbf{Bollob\'as and Scott} \cite{BollobasScottJudicious2002}]\label{conj:BS}
For $k\ge2$, every graph $G$ with $m$ edges and $n$ vertices has a partition into $k$ sets $V_1,\ldots,V_k$ such that
\begin{equation*}
\max_{1\le i<j\le k}e_G(V_i\cup V_j)\le \frac{4m}{k^2}+\frac{3(k-2)}{k^2}\sqrt{2m+\frac14}+\frac{(k-2)(2k-5)}{10}+O(n).
\end{equation*}
\end{conjecture}
Our main result treats unions of an arbitrary fixed number of parts and, at the same time, requires the partition to be balanced.
\begin{theorem}\label{thm:intro-general}
For $k\ge3$ and $1\le\ell\le k-1$, every graph $G$ with $n$ vertices and $m$ edges has a balanced partition into $k$ sets $V_1,\ldots,V_k$ such that
\begin{equation*}\label{eq:intro-general}
\max_{\substack{A\in\binom{[k]}{\ell}}}e_G\!\left(\bigcup_{i\in A}V_i\right)\le \frac{\ell^2}{k^2}m+\frac{\ell^2(k-\ell)}{k^2}(n-1)+\frac{\ell(k-\ell)}{k(k-1)}\sqrt{\left(\binom{k}{\ell}-1\right)m}.
\end{equation*}
\end{theorem}
 For $1\le \ell\le k-1$, the coefficient ${\ell^2(k-\ell)}/{k^2}$ of the linear term is asymptotically best possible among balanced $k$-partitions, as shown by $K_{\ell,n-\ell}$. The case $\ell=2$ gives the following explicit form of the Bollob\'as and Scott pair bound.
\begin{corollary}\label{cor:intro-BS}
For $k\ge2$, every graph $G$ with $n$ vertices and $m$ edges has a balanced partition into $k$ sets $V_1,\ldots,V_k$ such that
\begin{equation*}\label{eq:intro-BS}
\max_{1\le i<j\le k}e_G(V_i\cup V_j)\le \frac{4m}{k^2}+\frac{3(k-2)}{k^2}\sqrt{2m+\frac14}+\frac{(k-2)(2k-5)}{10}+C_kn,
\end{equation*}
where $C_k=\frac{k-2}{k^2}+\frac{k-2}{k(k-1)}\sqrt{k(k-1)-2}$.
\end{corollary}
The paper is organized as follows. Section~\ref{sec:Degree Blocks and Coloring} constructs the balanced block coloring and proves the required degree volume bound. Section~\ref{sec:Random Edge Counts} introduces the random edge-count decomposition. Section~\ref{sec:Block Estimates} estimates the block contributions and derives the variance bound. Section~\ref{sec:Proof of Main Theorem} combines these estimates to prove Theorem~\ref{thm:intro-general} and Corollary~\ref{cor:intro-BS}.
\section{Degree Blocks and Coloring}\label{sec:Degree Blocks and Coloring}
In this section, we construct the balanced block coloring and introduce the degree volumes of the color classes.  We also establish the degree bound that will be used in the later estimates.

Choose $r_0\in\{0,1,\ldots,k-1\}$ so that $N:=n+r_0$ is divisible by $k$, and let $\widehat G:=G\cup r_0K_1$. Thus $\widehat G$ is obtained from $G$ by adjoining $r_0$ isolated vertices, and $|V(\widehat G)|=N$. Order the vertices of $\widehat G$ as
\begin{equation}\label{eq:degree-order}
d_{\widehat G}(v_1)\ge d_{\widehat G}(v_2)\ge\cdots\ge d_{\widehat G}(v_N).
\end{equation}
Among vertices of degree zero, place the vertices of $V(G)$ before the vertices of $V(\widehat G)\setminus V(G)$. Since $k\mid N$, partition this ordering into consecutive blocks of size $k$: $B_s:=\{v_{(s-1)k+1},\ldots,v_{sk}\}$. Let $\mathcal B:=\{B_1,\ldots,B_{N/k}\}$ denote the resulting family of blocks. 

For each block $B_s\in\mathcal B$, independently choose a uniformly random bijection $\phi_s:B_s\to[k]$. Thus every color in $[k]$ occurs exactly once in each block. For $i\in[k]$, define \(\widetilde{V}_i\) be the color-\(i\) class in the enlarged graph. Since there are $N/k$ blocks and each block contains exactly one vertex of each color, we have $|\widetilde{V}_1| = \dots = |\widetilde{V}_k| = N/k$. Finally, for every $i\in[k]$, we obtain $V_i:=\widetilde V_i\cap V(G)$. By (\ref{eq:degree-order}), all vertices of $V(\widehat G)\setminus V(G)$ lie in the final block. Since $\phi_{N/k}$ is a bijection, these vertices receive distinct colors. Hence each color class loses at most one vertex when $V(\widehat G)\setminus V(G)$ is deleted, and therefore $|V_i|\in\left\{ N/k, N/k-1\right\}$.

For each $i\in[k]$, define the degree volume of the $i$th color class by
\begin{equation*}\label{eq:Di}
D_i:=\sum_{v\in\widetilde V_i}d_G(v)=\sum_{v\in V_i}d_G(v),
\end{equation*}
where the equality holds because every vertex in $V(\widehat G)\setminus V(G)$ is isolated. 

For $A\subseteq[k]$, define
\begin{equation*}
D_A:=\sum_{i\in A}D_i.
\end{equation*}
\begin{lemma}\label{lem:ell-volume}
	For every $A\subseteq[k]$ with $|A|=\ell$, we have
	\begin{equation*}
D_A\le\frac{2\ell m}{k}+\frac{\ell(k-\ell)}{k}(n-1).
	\end{equation*}
\end{lemma}
\begin{proof}[\bf Proof]
We first show that any two color classes have degree volumes differing by at most $n-1$. For each block $B_s=\{v_{(s-1)k+1},\ldots,v_{sk}\}$, let $x_{s,i}$ denote the unique vertex of $B_s$ receiving color $i$. Hence, for $i,j\in[k]$,
\begin{align}\label{eq:Di-Dj<n-1}
|D_i-D_j|&\le\sum_{s=1}^{N/k}\left|d_{\widehat G}(x_{s,i})-d_{\widehat G}(x_{s,j})\right|\notag\\&\le
\sum_{s=1}^{N/k}\left(d_{\widehat G}(v_{(s-1)k+1})-d_{\widehat G}(v_{sk})\right)\notag\\&=\sum_{s=1}^{N/k}\sum_{r=(s-1)k+1}^{sk-1}\left(d_{\widehat G}(v_r)-d_{\widehat G}(v_{r+1})\right)\notag\\&\le\sum_{r=1}^{N-1}\left(d_{\widehat G}(v_r)-d_{\widehat G}(v_{r+1})\right)\notag\\&=d_{\widehat G}(v_1)-d_{\widehat G}(v_N)\notag\\&\le n-1.
\end{align}
Now relabel $D_1,\ldots,D_k$ in nonincreasing order as $D_{(1)}\ge D_{(2)}\ge\cdots\ge D_{(k)}$. Since $|A|=\ell$, we have $D_A\le\sum_{r=1}^{\ell}D_{(r)}$. For every $r>\ell$, inequality~(\ref{eq:Di-Dj<n-1}) gives 
\begin{equation*}
D_{(r)}\ge D_{(1)}-(n-1)\geq\frac1\ell\sum_{s=1}^{\ell}D_{(s)}-(n-1),
\end{equation*}
because $D_{(1)}$ is the largest of the first $\ell$ terms. Therefore
\begin{equation*}
2m=\sum_{r=1}^{\ell}D_{(r)}+\sum_{r=\ell+1}^{k}D_{(r)}\ge\sum_{r=1}^{\ell}D_{(r)}+(k-\ell)
\left(\frac1\ell\sum_{r=1}^{\ell}D_{(r)}-(n-1)\right)=\frac{k}{\ell}\sum_{r=1}^{\ell}D_{(r)}-
(k-\ell)(n-1).
\end{equation*}
Then
\begin{equation*}
D_A\le\sum_{r=1}^{\ell}D_{(r)}\le\frac{2\ell m}{k}+\frac{\ell(k-\ell)}{k}(n-1),
\end{equation*}
as required.
\end{proof}
\section{Random Edge Counts}\label{sec:Random Edge Counts}
In this section, we decompose the induced edge count of a union of $\ell$ color classes into a deterministic part and a random term $Q_A$. We then split $Q_A$ into contributions from individual blocks and pairs of blocks, and record the basic properties needed later.

Fix $A\subseteq[k]$ with $|A|=\ell$.  For each
$v\in V(\widehat G)$, define
\begin{equation*}
X_v:=\mathbf 1_{\left\{v\in\bigcup_{i\in A}\widetilde V_i\right\}},\qquad\xi_v:=X_v-\frac{\ell}{k},\qquad Q_A:=\sum_{uv\in E(G)}\xi_u\xi_v.
\end{equation*}
\begin{lemma}
	For every $A\subseteq[k]$ with $|A|=\ell$, we have
	\begin{equation*}\label{lem:decomposition}
	e_G\left(\bigcup_{i\in A}V_i\right)=\frac{\ell}{k}D_A-\frac{\ell^2}{k^2}m+Q_A.
	\end{equation*}
\end{lemma}
\begin{proof}[\bf Proof]
	For every edge $uv\in E(G)$, the product $X_uX_v$ is the indicator that both endpoints belong to $\bigcup_{i\in A}V_i$. Hence
	\begin{align}\label{eq:eUVi}
	e_G\left(\bigcup_{i\in A}V_i\right)&=\sum_{uv\in E(G)}X_uX_v=\sum_{uv\in E(G)}\left(\frac{\ell^2}{k^2}+\frac{\ell}{k}(\xi_u+\xi_v)+\xi_u\xi_v\right)\notag\\&=
	\frac{\ell^2}{k^2}m+\frac{\ell}{k}\sum_{v\in V(G)}\left(d_G(v)\xi_v\right)+Q_A.
	\end{align}
	Moreover,
	\begin{equation*}
	\sum_{v\in V(G)}d_G(v)\xi_v=\sum_{v\in V(G)}d_G(v)\left(X_v-\frac{\ell}{k}\right)=\sum_{v\in \bigcup_{i\in A}\widetilde V_i}d_G(v)-\frac{\ell}{k}\sum_{v\in V(G)}d_G(v)=D_A-\frac{2\ell}{k}m.
	\end{equation*}
	Substituting (\ref{eq:eUVi}), we have
	\begin{equation*}
e_G\left(\bigcup_{i\in A}V_i\right)=\frac{\ell^2}{k^2}m+\frac{\ell}{k}\left(D_A-\frac{2\ell}{k}m
\right)+Q_A=\frac{\ell}{k}D_A-\frac{\ell^2}{k^2}m+Q_A.
	\end{equation*}
	This completes the proof.
\end{proof}
Independently in each block $B\in\mathcal B$, assign the colors $1,\ldots,k$ by a uniformly random bijection.  Throughout, $\mathbb P$ and $\mathbb E$ refer to this random block coloring.
We write $\mathbb E_B$ for expectation over the coloring of $B$ alone, and $\mathbb E_{B,C}$ for expectation over the independent colorings of two distinct blocks $B$ and $C$.
\begin{lemma}\label{lem:block-second}
For every vertex $u$, we have
\begin{equation*}
\E\xi_u^2=\frac{\ell(k-\ell)}{k^2}.
\end{equation*}
If $u\ne v$ lie in the same block, then
\begin{equation*}
\E(\xi_u\xi_v)=-\frac{\ell(k-\ell)}{k^2(k-1)}.
\end{equation*}
If $u$ and $v$ belong to different blocks, then $\E(\xi_u\xi_v)=0$. Consequently, $\mathbb E Q_A\le0$.
\end{lemma}
\begin{proof}[\bf Proof]
	Since $\mathbb P(X_u=1)={\ell}/{k}$ and $\xi_u=X_u-{\ell}/{k}$, we have $\mathbb E\xi_u=0$ and
	\begin{equation*}
	\mathbb E\xi_u^2=\operatorname{Var}(X_u)=\frac{\ell}{k}\left(1-\frac{\ell}{k}\right)=\frac{\ell(k-\ell)}{k^2}.
	\end{equation*}
	Now suppose that $u\ne v$ belong to the same block. Exactly $\ell$ of the $k$ vertices in that block receive colors from $A$. Hence
	\begin{equation*}
	\mathbb{P}(X_u = X_v = 1) = \mathbb{P}(X_u = 1) \mathbb{P}(X_v = 1 \mid X_u = 1) = \frac{\ell}{k} \cdot \frac{\ell - 1}{k - 1} = \frac{\ell(\ell - 1)}{k(k - 1)}.
	\end{equation*}
	Hence
	\begin{equation*}
	\mathbb E(\xi_u\xi_v)=\mathbb E\left(\left(X_u-\frac{\ell}{k}\right)\left(X_v-\frac{\ell}{k}\right)\right)=\mathbb E(X_uX_v)-\frac{\ell}{k}\mathbb E X_u-\frac{\ell}{k}\mathbb E X_v+\frac{\ell^2}{k^2}=-\frac{\ell(k-\ell)}{k^2(k-1)}.
	\end{equation*}
	If $u$ and $v$ belong to different blocks, then their block colorings are independent. Therefore $\mathbb E(\xi_u\xi_v)=\mathbb E\xi_u\,\mathbb E\xi_v=0$. Finally, by the definition of $Q_A$ and linearity of expectation,
	\begin{equation*}
	\mathbb E Q_A=\mathbb E\left(\sum_{uv\in E(G)}\xi_u\xi_v\right)=
	\sum_{uv\in E(G)}\mathbb E(\xi_u\xi_v)=-\frac{\ell(k-\ell)}{k^2(k-1)}m_0\le0,
	\end{equation*}
where $m_0=\left|\left\{uv\in E(G):u,v\text{ belong to the same block}\right\}\right|$. This completes the proof.
\end{proof}
For a block $B\in\calB$, define
\begin{equation*}
Q_B:=\sum_{\substack{uv\in E(G)\\u,v\in B}}\xi_u\xi_v.
\end{equation*}
For distinct blocks $B,C\in\mathcal B$, define
\begin{equation*}
Q_{BC}:=\sum_{\substack{uv\in E(G)\\u\in B,\ v\in C}}\xi_u\xi_v.
\end{equation*}
Since every edge of $G$ is either contained in one block or joins two distinct blocks,
\begin{equation}\label{eq:block-decomp}
Q_A=\sum_{B\in\calB}Q_B+\sum_{B<C}Q_{BC}.
\end{equation}
For $B\in\mathcal B$, let $\mathbb E_B$ denote expectation with respect to the random coloring of $B$, with all other block colorings fixed.  For distinct blocks $B,C$, let $\mathbb E_{B,C}$ denote expectation with respect to their independent random colorings.
\begin{lemma}\label{lem:EBCQBC=0}
	If $B\ne C$, then, for every fixed coloring of $B$, we have $\mathbb E_C Q_{BC}=0$. Similarly, for every fixed coloring of $C$, we have $\mathbb E_B Q_{BC}=0$. Consequently, $\mathbb E_{B,C} Q_{BC}=0$.
\end{lemma}
\begin{proof}[\bf Proof]
	Fix the coloring of $B$.  Then every $\xi_u$ with $u\in B$ is fixed,
	while the coloring of $C$ remains uniformly random. Hence 
	\begin{equation*}
\mathbb E_C Q_{BC}=\mathbb E_C\left(\sum_{\substack{uv\in E(G)\\u\in B,\ v\in C}}\xi_u\xi_v\right)=\sum_{\substack{uv\in E(G)\\u\in B,\ v\in C}}\xi_u\,\mathbb E_C\xi_v=0,
	\end{equation*}
because $\mathbb E_C\xi_v=0$ for every $v\in C$ by Lemma~\ref{lem:block-second}.
Interchanging $B$ and $C$ gives $\mathbb E_B Q_{BC}=0$. Finally, since the colorings of $B$ and $C$ are independent, we have $\mathbb E_{B,C}Q_{BC}=\mathbb E_B\bigl(\mathbb E_C Q_{BC}\bigr)=0$. This completes the proof.
\end{proof}
Lemma~\ref{lem:EBCQBC=0}, together with the independence of the block colorings, implies that distinct terms in the decomposition~\eqref{eq:block-decomp} have no covariance contribution.  This gives the following variance decomposition.
\begin{lemma}\label{lem:orthogonality}
Any two distinct summands on the right-hand side of \eqref{eq:block-decomp} are uncorrelated. Consequently,
\begin{equation*}\label{eq:variance-decomp}
\Var(Q_A)=\sum_{B\in\calB}\Var(Q_B)+\sum_{B<C}\Var(Q_{BC}).
\end{equation*}
\end{lemma}

\begin{proof}[\bf Proof]
	We consider the possible pairs of distinct summands in~\eqref{eq:block-decomp}. If two summands depend on disjoint sets of blocks, then they are independent. Suppose next that two cross-block terms share exactly one block, say $Q_{BC}$ and $Q_{BD}$.  Fix the coloring of $B$. The remaining randomness comes from the independent colorings of $C$ and $D$. By Lemma~\ref{lem:EBCQBC=0}, we have
	\begin{equation*}
\mathbb E_{B,C,D}(Q_{BC}Q_{BD})=\mathbb E_B\left(\mathbb E_{C,D}(Q_{BC}Q_{BD})\right)=\mathbb E_B
\left((\mathbb E_CQ_{BC})\cdot(\mathbb E_DQ_{BD})\right)=0.
	\end{equation*}
Since $\mathbb E Q_{BC}=\mathbb E Q_{BD}=0$, it follows that 
\begin{equation*}
\operatorname{Cov}(Q_{BC},Q_{BD})=\mathbb E(Q_{BC}Q_{BD})-\mathbb E Q_{BC}\,\mathbb E Q_{BD}=0.
\end{equation*}
Finally, consider $Q_B$ and $Q_{BC}$.  Once the coloring of $B$ is fixed, $Q_B$ is fixed, by Lemma~\ref{lem:EBCQBC=0}
\begin{equation*}
\mathbb E_{B,C}(Q_BQ_{BC})=\mathbb E_B\left(Q_B\,\mathbb E_CQ_{BC}\right)=0.
\end{equation*}
Since $\mathbb E Q_{BC}=0$, we have $\operatorname{Cov}(Q_B,Q_{BC})=0$. These cases exhaust all pairs of distinct summands in~\eqref{eq:block-decomp}. Thus all cross-covariances vanish.  Therefore,
\begin{equation*}
\operatorname{Var}(Q_A)=\sum_{B\in\mathcal B}\operatorname{Var}(Q_B)+\sum_{B<C}\operatorname{Var}(Q_{BC}).
\end{equation*}
This completes the proof.
\end{proof}
\section{Block Estimates}\label{sec:Block Estimates}
In this section, we estimate the contributions from individual blocks and pairs of blocks.  We first study the corresponding block variables, then combine the resulting estimates to obtain a uniform bound for $Q_A$.

Fix an ordering $B=\{w_1,\ldots,w_k\}$ of a block, and write $\xi_B:=(\xi_{w_1},\ldots,\xi_{w_k})^{\mathsf T}$. Let $J$ denote the $k\times k$ all-ones matrix.
\begin{lemma}\label{lem:cov-matrix}
	For every block $B$, we have
	\begin{equation}\label{eq:EBBT}
\mathbb E(\xi_B\xi_B^{\mathsf T})=\frac{\ell(k-\ell)}{k(k-1)}\left(I-\frac1kJ\right).
	\end{equation}
	Moreover, $\sum_{v\in B}\xi_v=0$ for every coloring outcome.
\end{lemma}
\begin{proof}[\bf Proof]
The $(r,s)$-entry of $\mathbb E(\xi_B\xi_B^{\mathsf T})$ is $\mathbb E(\xi_{w_r}\xi_{w_s})$. By Lemma~\ref{lem:block-second}, we have
\begin{equation*}
\mathbb E(\xi_{w_r}\xi_{w_s})=\begin{cases}\dfrac{\ell(k-\ell)}{k^2},& r=s,\\[3mm]-\dfrac{\ell(k-\ell)}{k^2(k-1)},& r\ne s.
\end{cases}
\end{equation*}
	On the other hand, the matrix $	\frac{\ell(k-\ell)}{k(k-1)}\left(I-\frac1kJ\right)$ has diagonal entries
	\begin{equation*}
\frac{\ell(k-\ell)}{k(k-1)}\left(1-\frac1k\right)=\frac{\ell(k-\ell)}{k(k-1)}\frac{k-1}{k}=\frac{\ell(k-\ell)}{k^2},
	\end{equation*}
	and off-diagonal entries
	\begin{equation*}
-\frac{\ell(k-\ell)}{k(k-1)}\frac1k=-\frac{\ell(k-\ell)}{k^2(k-1)},
	\end{equation*}
	which proves~(\ref{eq:EBBT}). Finally, every block contains exactly $\ell$ vertices whose colors belong to $A$. Therefore $\sum_{v\in B}X_v=\ell$. Since $\xi_v=X_v-{\ell}/{k}$, we obtain
	\begin{equation*}
\sum_{v\in B}\xi_v=\sum_{v\in B}\left(X_v-\frac{\ell}{k}\right)=\sum_{v\in B}X_v-k\frac{\ell}{k}=\ell-\ell=0.
	\end{equation*}
This completes the proof.
\end{proof}
For distinct blocks $B,C$, fix orderings $B=\{u_1,\ldots,u_k\}$, $C=\{v_1,\ldots,v_k\}$, and define the $k\times k$ matrix $M_{BC}=(m_{rs})_{1\le r,s\le k}$ by $m_{rs}:=\mathbf 1_{\{u_rv_s\in E(G)\}}$. Thus $M_{BC}$ is the bipartite adjacency matrix of the edges between $B$ and $C$. Writing $\xi_B=(\xi_{u_1},\ldots,\xi_{u_k})^{\mathsf T}$ and $\xi_C=(\xi_{v_1},\ldots,\xi_{v_k})^{\mathsf T}$, we have
\begin{equation*}
\xi_B^{\mathsf T}M_{BC}\xi_C=\sum_{r=1}^k\sum_{s=1}^km_{rs}\xi_{u_r}\xi_{v_s}=\sum_{\substack{1\le r,s\le k\\u_rv_s\in E(G)}}
\xi_{u_r}\xi_{v_s}=\sum_{\substack{uv\in E(G)\\u\in B,\ v\in C}}\xi_u\xi_v=Q_{BC}.
\end{equation*}
\begin{lemma}\label{lem:cross-variance}
	For distinct blocks $B,C$, we have
	\begin{equation*}
\operatorname{Var}(Q_{BC})=\left(\frac{\ell(k-\ell)}{k(k-1)}\right)^2\left\|\left(I-\frac1kJ\right)M_{BC}\left(I-\frac1kJ\right)\right\|_\mathrm{F}^2\le\left(\frac{\ell(k-\ell)}{k(k-1)}\right)^2e_G(B,C),
	\end{equation*}
\end{lemma}
\begin{proof}[\bf Proof]
	By Lemma~\ref{lem:EBCQBC=0}, we have $\mathbb E_{B,C}Q_{BC}=0$ and hence $\operatorname{Var}(Q_{BC})=\mathbb E_{B,C}Q_{BC}^2$. Fix the coloring of $B$. Since $Q_{BC}=\xi_B^{\mathsf T}M_{BC}\xi_C$, we have
	\begin{align*}
\mathbb E_C Q_{BC}^2&=\mathbb E_C\left(\xi_B^{\mathsf T}M_{BC}\xi_C\xi_C^{\mathsf T}M_{BC}^{\mathsf T}\xi_B\right)\\&=\xi_B^{\mathsf T}M_{BC}\mathbb E_C(\xi_C\xi_C^{\mathsf T})M_{BC}^{\mathsf T}\xi_B\\&=\frac{\ell(k-\ell)}{k(k-1)}\xi_B^{\mathsf T}M_{BC}\left(I-\frac1kJ\right)M_{BC}^{\mathsf T}\xi_B.
	\end{align*}
	Taking expectation over the coloring of $B$ and using Lemma~\ref{lem:cov-matrix} gives
	\begin{align*}
\operatorname{Var}(Q_{BC})&=\frac{\ell(k-\ell)}{k(k-1)}\mathbb E_B\left(\xi_B^{\mathsf T}M_{BC}\left(I-\frac1kJ\right)M_{BC}^{\mathsf T}\xi_B\right)\\&=\frac{\ell(k-\ell)}{k(k-1)}\operatorname{tr}\left(M_{BC}\left(I-\frac1kJ\right)M_{BC}^{\mathsf T}\mathbb E_B(\xi_B\xi_B^{\mathsf T})\right)\\&=\left(\frac{\ell(k-\ell)}{k(k-1)}\right)^2\operatorname{tr}\left(M_{BC}\left(I-\frac1kJ\right)M_{BC}^{\mathsf T}\left(I-\frac1kJ\right)\right).
	\end{align*}
	Since $\left(I-J/k\right)^{\mathsf T}=I-J/k$ and $\left(I-J/k\right)^2=I-J/k$, we have
	\begin{equation}\label{eq:VarQBC}
	\operatorname{Var}(Q_{BC})=\left(\frac{\ell(k-\ell)}{k(k-1)}\right)^2\left\|\left(I-\frac1kJ\right)M_{BC}\left(I-\frac1kJ\right)\right\|_\mathrm{F}^2.
	\end{equation}
	For every vector $x=(x_1,\ldots,x_k)^{\mathsf T}$, we have
	\begin{align*}
\left\|\left(I-\frac1kJ\right)x\right\|_2^2&=\sum_{r=1}^k\left(x_r-\frac1k\sum_{s=1}^k x_s\right)^2\\&=\sum_{r=1}^k x_r^2-\frac{2}{k}\sum_{r=1}^kx_r\sum_{s=1}^k x_s+\frac{1}{k^2}\sum_{r=1}^k\left(\sum_{s=1}^k x_s\right)^2\\&=\sum_{r=1}^k x_r^2-\frac{2}{k}\left(\sum_{r=1}^k x_r\right)^2+\frac{k}{k^2}\left(\sum_{r=1}^k x_r\right)^2\\&=\|x\|_2^2-\frac1k\left(\sum_{r=1}^k x_r\right)^2\le\|x\|_2^2.
\end{align*}
Applying this inequality to each column of $M_{BC}$ gives
\begin{equation*}
\left\|\left(I-\frac1kJ\right)M_{BC}\right\|_\mathrm{F}^2=\sum_{s=1}^k\left\|\left(I-\frac1kJ\right)(M_{BC})_{\ast s}\right\|_2^2\le\sum_{s=1}^k\left\|(M_{BC})_{\ast s}\right\|_2^2=\|M_{BC}\|_\mathrm{F}^2,
\end{equation*}
where $(M_{BC})_{\ast s}$ denotes the $s$th column of $M_{BC}$. Since $I-J/k$ is symmetric, applying the same inequality to the rows gives
\begin{equation*}
\left\|M_{BC}\left(I-\frac1kJ\right)\right\|_\mathrm{F}=\left\|\left(I-\frac1kJ\right)M_{BC}^\mathsf{T}\right\|_\mathrm{F}\le\|M_{BC}^\mathsf{T}\|_\mathrm{F}=\|M_{BC}\|_\mathrm{F}.
\end{equation*}
Therefore
\begin{equation*}
\left\|\left(I-\frac1kJ\right)M_{BC}\left(I-\frac1kJ\right)\right\|_\mathrm{F}\le\left\|M_{BC}\left(I-\frac1kJ\right)\right\|_\mathrm{F}\le\|M_{BC}\|_\mathrm{F}.
\end{equation*}
Finally, since $M_{BC}=(m_{rs})$ is the bipartite adjacency matrix
between $B$ and $C$, we have
\begin{equation*}
\|M_{BC}\|_\mathrm{F}^2=\sum_{r=1}^k\sum_{s=1}^k m_{rs}^2=\sum_{r=1}^k\sum_{s=1}^k m_{rs}=e_G(B,C).
\end{equation*}
Hence
\begin{equation}\label{eq:<eBC}
\left\|\left(I-\frac1kJ\right)M_{BC}\left(I-\frac1kJ\right)\right\|_\mathrm{F}^2\le e_G(B,C).
\end{equation}
Combining (\ref{eq:VarQBC}) and (\ref{eq:<eBC}), this completes the proof.
\end{proof}
Fix a block $B$.  Let $S:=\{x\in B:\text{ the color of }x\text{ belongs to }A\}$. Since the coloring of $B$ is a uniformly random bijection and $|A|=\ell$, the set $S$ is uniformly distributed over all $\ell$-subsets of $B$.
\begin{lemma}\label{lem:subset-representation}
For every block $B$, we have
\begin{equation*}\label{eq:subset-representation}
Q_B=e_G(S)-\frac{\ell}{k}\sum_{x\in S}d_{G[B]}(x)+\frac{\ell^2}{k^2}e_G(B).
\end{equation*}
\end{lemma}
\begin{proof}[\bf Proof]
	For every $x\in B$, we have $X_x=\mathbf 1_{\{x\in S\}}$ and $\xi_x=X_x-\frac{\ell}{k}$. Hence
	\begin{equation*}
Q_B=\sum_{xy\in E(G[B])}\xi_x\xi_y=\sum_{xy\in E(G[B])}X_xX_y-\frac{\ell}{k}\sum_{xy\in E(G[B])}(X_x+X_y)+\frac{\ell^2}{k^2}\sum_{xy\in E(G[B])}1.
	\end{equation*}
	Since $X_xX_y=1$ exactly when both $x$ and $y$ belong to $S$, we have $\sum_{xy\in E(G[B])}X_xX_y=e_G(S)$. Moreover, each $X_x$ occurs once for every edge of $G[B]$
	incident with $x$.  Therefore
	\begin{equation*}
	\sum_{xy\in E(G[B])}(X_x+X_y)=\sum_{x\in B}d_{G[B]}(x)X_x=\sum_{x\in S}d_{G[B]}(x).
	\end{equation*}
	Finally, $\sum_{xy\in E(G[B])}1=e_G(B)$. Substituting these identities gives
	\begin{equation*}
Q_B=e_G(S)-\frac{\ell}{k}\sum_{x\in S}d_{G[B]}(x)+\frac{\ell^2}{k^2}e_G(B).
	\end{equation*}
	This completes the proof.
\end{proof}
For $r\ge 1$, write $(x)_r:=x(x-1)\cdots(x-r+1)$. Since $S$ is a uniformly random $\ell$-subset of the $k$-vertex block $B$, for any distinct vertices $x_1,\ldots,x_r\in B$, we have
\begin{equation*}
\mathbb P(x_1,\ldots,x_r\in S)=\frac{\binom{k-r}{\ell-r}}{\binom{k}{\ell}}=\frac{(\ell)_r}{(k)_r},
\end{equation*}
with the probability interpreted as zero when $r>\ell$.
\begin{lemma}\label{lem:subset-moments}
Assume $k\ge4$. Then
\begin{equation}\label{eq:EeS}
\mathbb E_S\,e_G(S)=\frac{(\ell)_2}{(k)_2}e_G(B),
\end{equation}
and
\begin{equation}\label{eq:EdB}
\mathbb E_S\left(\sum_{x\in S}d_{G[B]}(x)\right)=\frac{2\ell}{k}e_G(B).
\end{equation}
Moreover,
\begin{equation}\label{eq:EeS2}
\mathbb E_S\,e_G(S)^2=\frac{(\ell)_4}{(k)_4}e_G(B)^2+\left(\frac{(\ell)_2}{(k)_2}-2\frac{(\ell)_3}{(k)_3}+\frac{(\ell)_4}{(k)_4}\right)e_G(B)+\left(\frac{(\ell)_3}{(k)_3}-\frac{(\ell)_4}{(k)_4}\right)\sum_{x\in B}d_{G[B]}(x)^2,
\end{equation}
\begin{equation}\label{eq:EeSdb}
\mathbb E_S\left(e_G(S)\sum_{x\in S}d_{G[B]}(x)\right)=2\frac{(\ell)_3}{(k)_3}e_G(B)^2+\left(\frac{(\ell)_2}{(k)_2}-\frac{(\ell)_3}{(k)_3}\right)\sum_{x\in B}d_{G[B]}(x)^2,
\end{equation}
and
\begin{equation}\label{eq:EdB2}
\mathbb E_S\left(\sum_{x\in S}d_{G[B]}(x)\right)^2=4\frac{(\ell)_2}{(k)_2}e_G(B)^2+\left(\frac{\ell}{k}-\frac{(\ell)_2}{(k)_2}\right)\sum_{x\in B}d_{G[B]}(x)^2.
\end{equation}
\end{lemma}
\begin{proof}[\bf Proof]
For a fixed edge $xy\in E(G[B])$, the edge is contained in $S$
exactly when both $x$ and $y$ belong to $S$. Hence
\begin{equation*}
\mathbb E_S\,e_G(S)=\sum_{xy\in E(G[B])}\mathbb P(x,y\in S)=\frac{\ell(\ell-1)}{k(k-1)}e_G(B),
\end{equation*}
which proves~(\ref{eq:EeS}). Similarly,
\begin{equation*}
\mathbb E_S\left(\sum_{x\in S}d_{G[B]}(x)\right)=\sum_{x\in B}d_{G[B]}(x)\,\mathbb P(x\in S)=\frac{\ell}{k}\sum_{x\in B}d_{G[B]}(x)=\frac{2\ell}{k}e_G(B),
\end{equation*}
which proves~(\ref{eq:EdB}). We next compute $\mathbb E_Se_G(S)^2$. Since $e_G(S)=\sum_{xy\in E(G[B])}\mathbf 1_{\{x,y\in S\}}$ expanding the square and separating pairs of distinct edges
according as they are adjacent or disjoint gives
\begin{align}\label{eq:EeS^2}
\mathbb E_S e_G(S)^2&=\sum_{xy\in E(G[B])}\mathbb P(x,y\in S)+2\sum_{\substack{\{e,f\}\subseteq E(G[B])\\ |e\cap f|=1}}\mathbb P(e\cup f\subseteq S)+2\sum_{\substack{\{e,f\}\subseteq E(G[B])\\ e\cap f=\emptyset}}\mathbb P(e\cup f\subseteq S)\notag\\&=\frac{(\ell)_2}{(k)_2}e_G(B)+
\frac{2(\ell)_3}{(k)_3}\sum_{x\in B}\binom{d_{G[B]}(x)}{2}+\frac{2(\ell)_4}{(k)_4}
\left(\binom{e_G(B)}{2}-\sum_{x\in B}\binom{d_{G[B]}(x)}{2}\right).
\end{align}
Then
\begin{equation}\label{eq:2sumbidx2}
2\sum_{x\in B}\binom{d_{G[B]}(x)}{2}=\sum_{x\in B}d_{G[B]}(x)\bigl(d_{G[B]}(x)-1\bigr)
=\sum_{x\in B}d_{G[B]}(x)^2-2e_G(B),
\end{equation}
and hence
\begin{align}\label{eq:2eb-dx}
2\left(\binom{e_G(B)}{2}-\sum_{x\in B}\binom{d_{G[B]}(x)}{2}\right)&=e_G(B)\bigl(e_G(B)-1\bigr)-
\left(\sum_{x\in B}d_{G[B]}(x)^2-2e_G(B)\right)\notag\\&=e_G(B)^2+e_G(B)-\sum_{x\in B}d_{G[B]}(x)^2.
\end{align}
Substituting~\eqref{eq:2sumbidx2} and~\eqref{eq:2eb-dx} into~(\ref{eq:EeS^2}), we obtain
\begin{align*}
&\qquad\mathbb E_S\,e_G(S)^2\\&=\frac{(\ell)_2}{(k)_2}e_G(B)+\frac{(\ell)_3}{(k)_3}\left(\sum_{x\in B}d_{G[B]}(x)^2-2e_G(B)\right)+\frac{(\ell)_4}{(k)_4}\left(e_G(B)^2+e_G(B)-\sum_{x\in B}d_{G[B]}(x)^2\right)\\&=\frac{(\ell)_4}{(k)_4}e_G(B)^2+\left(\frac{(\ell)_2}{(k)_2}-2\frac{(\ell)_3}{(k)_3}+\frac{(\ell)_4}{(k)_4}\right)e_G(B)+\left(\frac{(\ell)_3}{(k)_3}-\frac{(\ell)_4}{(k)_4}\right)\sum_{x\in B}d_{G[B]}(x)^2,
\end{align*}
which proves (\ref{eq:EeS2}). We next compute $\mathbb E_S\left(e_G(S)\sum_{x\in S}d_{G[B]}(x)\right)$. Since 
\begin{equation}\label{eq:eS,dB}
e_G(S)=\sum_{xy\in E(G[B])}\mathbf 1_{\{x,y\in S\}},\qquad\sum_{z\in S}d_{G[B]}(z)=\sum_{z\in B}d_{G[B]}(z)\mathbf 1_{\{z\in S\}},
\end{equation}
we have
\begin{align*}
&\qquad\mathbb E_S\left(e_G(S)\sum_{z\in S}d_{G[B]}(z)\right)\\&=\sum_{xy\in E(G[B])}\sum_{z\in B}
d_{G[B]}(z)\,\mathbb E\left(\mathbf 1_{\{x,y\in S\}}\mathbf 1_{\{z\in S\}}\right)\\&=
\sum_{xy\in E(G[B])}\left(\frac{(\ell)_2}{(k)_2}\bigl(d_{G[B]}(x)+d_{G[B]}(y)\bigr)+\frac{(\ell)_3}{(k)_3}\sum_{z\in B\setminus\{x,y\}}d_{G[B]}(z)\right)\\&=\sum_{xy\in E(G[B])}
\left(\frac{(\ell)_2}{(k)_2}\bigl(d_{G[B]}(x)+d_{G[B]}(y)\bigr)+\frac{(\ell)_3}{(k)_3}\bigl(2e_G(B)-d_{G[B]}(x)-d_{G[B]}(y)\bigr)\right)\\&=2\frac{(\ell)_3}{(k)_3}e_G(B)^2+\left(\frac{(\ell)_2}{(k)_2}-\frac{(\ell)_3}{(k)_3}\right)\sum_{xy\in E(G[B])}\bigl(d_{G[B]}(x)+d_{G[B]}(y)\bigr).
\end{align*}
Moreover,
\begin{equation*}
\sum_{xy\in E(G[B])}\left(d_{G[B]}(x)+d_{G[B]}(y)\right)=\sum_{x\in B}d_{G[B]}(x)\,d_{G[B]}(x)
=\sum_{x\in B}d_{G[B]}(x)^2.
\end{equation*}
Hence
\begin{equation*}
\mathbb E_S\left(e_G(S)\sum_{z\in S}d_{G[B]}(z)\right)=2\frac{(\ell)_3}{(k)_3}e_G(B)^2+\left(
\frac{(\ell)_2}{(k)_2}-\frac{(\ell)_3}{(k)_3}\right)\sum_{x\in B}d_{G[B]}(x)^2,
\end{equation*}
which proves (\ref{eq:EeSdb}). By (\ref{eq:eS,dB}), we have
\begin{equation*}
\left(\sum_{x\in S}d_{G[B]}(x)\right)^2=\sum_{x\in B}d_{G[B]}(x)^2\mathbf 1_{\{x\in S\}}
+2\sum_{\substack{x,y\in B\\x<y}}d_{G[B]}(x)d_{G[B]}(y)\mathbf 1_{\{x,y\in S\}}.
\end{equation*}
Taking expectations gives
\begin{equation*}
\mathbb E_S\left(\sum_{x\in S}d_{G[B]}(x)\right)^2=\frac{\ell}{k}\sum_{x\in B}d_{G[B]}(x)^2
+2\frac{\ell(\ell-1)}{k(k-1)}\sum_{x<y}d_{G[B]}(x)d_{G[B]}(y).
\end{equation*}
Since
\begin{equation*}
2\sum_{x<y}d_{G[B]}(x)d_{G[B]}(y)=\left(\sum_{x\in B}d_{G[B]}(x)\right)^2-\sum_{x\in B}d_{G[B]}(x)^2=4e_G(B)^2-\sum_{x\in B}d_{G[B]}(x)^2,
\end{equation*}
we obtain
\begin{equation*}
\mathbb E_S\left(\sum_{x\in S}d_{G[B]}(x)\right)^2=4\frac{\ell(\ell-1)}{k(k-1)}e_G(B)^2
+\left(\frac{\ell}{k}-\frac{\ell(\ell-1)}{k(k-1)}\right)\sum_{x\in B}d_{G[B]}(x)^2,
\end{equation*}
which proves~(\ref{eq:EdB2}). This completes the proof.
\end{proof}
The above results can now be substituted into the representation of $Q_B$ to obtain an exact expression for its variance.
\begin{lemma}\label{lem:exact-within}
Assume $k\ge4$, and put $t:=\ell(k-\ell)$. Then, for every block $B$,
	\begin{align}\label{eq:exact-within}
\operatorname{Var}(Q_B)=&\frac{t(t-k+1)}{k(k-1)(k-2)(k-3)}\,e_G(B)+\frac{2t\bigl(t(k^2+10k-12)-3k^2(k-1)\bigr)}{k^4(k-1)^2(k-2)(k-3)}\,e_G(B)^2\notag\\&+\frac{t\bigl(2k^2-t(k+6)\bigr)}{k^3(k-1)(k-2)(k-3)}\sum_{x\in B}d_{G[B]}(x)^2.
	\end{align}
\end{lemma}
\begin{proof}[\bf Proof]
	Fixed block $B$. By Lemma~\ref{lem:subset-representation}, we have
	\begin{align*}
\operatorname{Var}(Q_B)&=\operatorname{Var}\left(e_G(S)-\frac{\ell}{k}\sum_{x\in S}d_{G[B]}(x)\right)\\&=\mathbb E_S e_G(S)^2-\frac{2\ell}{k}\mathbb E_S\left(e_G(S)\sum_{x\in S}d_{G[B]}(x)\right)+\frac{\ell^2}{k^2}\mathbb E_S\left(\sum_{x\in S}d_{G[B]}(x)\right)^2\\&\qquad-\left(\mathbb E_S e_G(S)-\frac{\ell}{k}\mathbb E_S\left(\sum_{x\in S}d_{G[B]}(x)\right)\right)^2.
	\end{align*}
	Combining (\ref{eq:EeS}) and (\ref{eq:EdB}), we have
	\begin{equation*}
\left(\mathbb E_S e_G(S)-\frac{\ell}{k}\mathbb E_S\left(\sum_{x\in S}d_{G[B]}(x)\right)\right)^2
=\left(\frac{(\ell)_2}{(k)_2}-\frac{2\ell^2}{k^2}\right)^2e_G(B)^2.
	\end{equation*}
By Lemma~\ref{lem:subset-moments}, we have
\begin{align*}
&\qquad\operatorname{Var}(Q_B)\\&=\left(\frac{(\ell)_2}{(k)_2}-2\frac{(\ell)_3}{(k)_3}+\frac{(\ell)_4}{(k)_4}\right)e_G(B)+\left(\frac{(\ell)_4}{(k)_4}-\frac{4\ell}{k}\frac{(\ell)_3}{(k)_3}+\frac{4\ell^2}{k^2}\frac{(\ell)_2}{(k)_2}-\left(\frac{(\ell)_2}{(k)_2}-\frac{2\ell^2}{k^2}\right)^2\right)e_G(B)^2\\&\qquad+\left(\frac{(\ell)_3}{(k)_3}-\frac{(\ell)_4}{(k)_4}-\frac{2\ell}{k}\left(\frac{(\ell)_2}{(k)_2}-\frac{(\ell)_3}{(k)_3}\right)+\frac{\ell^2}{k^2}\left(\frac{\ell}{k}-\frac{(\ell)_2}{(k)_2}\right)\right)\sum_{x\in B}d_{G[B]}(x)^2\\&=\frac{t(t-k+1)}{k(k-1)(k-2)(k-3)}\,e_G(B)+\frac{2t\bigl(t(k^2+10k-12)-3k^2(k-1)\bigr)}{k^4(k-1)^2(k-2)(k-3)}\,e_G(B)^2\\&\qquad+\frac{t\bigl(2k^2-t(k+6)\bigr)}{k^3(k-1)(k-2)(k-3)}\sum_{x\in B}d_{G[B]}(x)^2,
\end{align*}
where $t:=\ell(k-\ell)$. This completes the proof.
\end{proof}
The exact variance still involves the sum of squared degrees inside the block. We will use the following standard bounds for this.
\begin{lemma}\label{lem:degree-square}
For every block $B$, we have
\begin{equation}\label{eq:degree-square}
\max\left\{2e_G(B),\frac{4e_G(B)^2}{k}\right\}\le\sum_{x\in B}d_{G[B]}(x)^2\le e_G(B)(k-2)+\frac{2e_G(B)^2}{k-1}.
\end{equation}
\end{lemma}
\begin{proof}[\bf Proof]
	Since $d_{G[B]}(x)$ is a nonnegative integer, we have $d_{G[B]}(x)^2\ge d_{G[B]}(x)$ for every $x\in B$.  Hence,
	\begin{equation}\label{eq:dB=2e}
	\sum_{x\in B}d_{G[B]}(x)^2\ge\sum_{x\in B}d_{G[B]}(x)\\=2e_G(B).
	\end{equation}
	On the other hand, the Cauchy--Schwarz inequality gives
	\begin{equation*}
\left(\sum_{x\in B}d_{G[B]}(x)\right)^2\le\left(\sum_{x\in B}1^2\right)\left(\sum_{x\in B}d_{G[B]}(x)^2\right)=k\sum_{x\in B}d_{G[B]}(x)^2.
	\end{equation*}
Using (\ref{eq:dB=2e}), we obtain
\begin{equation}\label{eq:db>4e/k}
\sum_{x\in B}d_{G[B]}(x)^2\ge\frac{4e_G(B)^2}{k}.
\end{equation}
	Combining (\ref{eq:dB=2e}) and (\ref{eq:db>4e/k}), we obtain
	\begin{equation*}
	\sum_{x\in B}d_{G[B]}(x)^2\ge\max\left\{2e_G(B),\frac{4e_G(B)^2}{k}\right\}.
	\end{equation*}
For the upper bound, applying de Caen's~\cite{deCaen1998} inequality to the $k$-vertex simple graph $G[B]$ gives
	\begin{equation*}
\sum_{x\in B}d_{G[B]}(x)^2\le e_G(B)\left(k-2+\frac{2e_G(B)}{k-1}\right)=e_G(B)(k-2)+\frac{2e_G(B)^2}{k-1}.
	\end{equation*}
	This proves~\eqref{eq:degree-square}.
\end{proof}
Combining the exact formula of Lemma~\ref{lem:exact-within} with the bounds in Lemma~\ref{lem:degree-square}, we now derive a uniform estimate for the contribution of a single block.
\begin{lemma}\label{lem:within-linear}
For every block $B$, we have
\begin{equation}\label{eq:within-linear}
\operatorname{Var}(Q_B)\le\left(\frac{\ell(k-\ell)}{k(k-1)}\right)^2e_G(B).
\end{equation}
\end{lemma}
\begin{proof}[\bf Proof]
We first consider the case $k=3$. Then $\ell\in\{1,2\}$. We show explicitly that the two values of $\ell$ are equivalent. 

Let $A\subseteq[3]$ with $|A|=1$, and let $A^c=[3]\setminus A$,
so that $|A^c|=2$.  For $x\in B$, write $X_x^A:=\mathbf 1_{\{\text{the color of }x\text{ belongs to }A\}}$ and $X_x^{A^c}:=\mathbf 1_{\{\text{the color of }x\text{ belongs to }A^c\}}$. Since every color belongs to exactly one of $A$ and $A^c$, we have $X_x^{A^c}=1-X_x^A$. For the one-color set $A$, we have $\xi_x^A=X_x^A-1/3$ whereas for its two-color complement $A^c$,
\begin{equation*}
\xi_x^{A^c}=X_x^{A^c}-\frac23=1-X_x^A-\frac23=-\left(X_x^A-\frac13\right)=-\xi_x^A.
\end{equation*}
Hence, for every edge $xy\in E(G[B])$, we have $\xi_x^{A^c}\xi_y^{A^c}=(-\xi_x^A)(-\xi_y^A)=\xi_x^A\xi_y^A$. Therefore
\begin{equation*}
Q_B(A^c)=\sum_{xy\in E(G[B])}\xi_x^{A^c}\xi_y^{A^c}=\sum_{xy\in E(G[B])}\xi_x^A\xi_y^A=Q_B(A).
\end{equation*}
Consequently, $\operatorname{Var}(Q_B(A^c))=\operatorname{Var}(Q_B(A))$. Thus the cases $\ell=1$ and $\ell=2$ are identical, and it is enough to treat $\ell=1$. 

Since $S$ is a uniformly random one-element subset of $B$, exactly one vertex of $B$ belongs to $S$.  For every $x\in B$, we have
\begin{equation}\label{eq:epslionx}
\xi_x=X_x-\frac13=\begin{cases}\dfrac23, & x\in S,\\[2mm]-\dfrac13, & x\notin S.
\end{cases}
\end{equation}
Since $|B|=3$, we have $e_G(B)\in\{0,1,2,3\}$.

\item \textbf{Case 1.} If $e_G(B)=0$, then $E(G[B])=\varnothing$, and therefore
\begin{equation*}
Q_B=\sum_{xy\in E(G[B])}\xi_x\xi_y=0.
\end{equation*}
Hence $\operatorname{Var}(Q_B)=0$.

\item  \textbf{Case 2.} If $e_G(B)=3$, then $G[B]=K_3$. For every coloring outcome, exactly one vertex belongs to $S$. By (\ref{eq:epslionx}), we have
\begin{equation*}
Q_B=\sum_{xy\in E(G[B])}\xi_x\xi_y=\left(\frac23\right)\left(-\frac13\right)+\left(\frac23\right)\left(-\frac13\right)+\left(-\frac13\right)\left(-\frac13\right)=-\frac13.
\end{equation*} 
Thus $Q_B$ is constant, and hence $\operatorname{Var}(Q_B)=0$.

\item \textbf{Case 3.} $e_G(B)=1$. Let $xy$ be the unique edge of $G[B]$, and let $z$ be the unique vertex in $B\setminus\{x,y\}$. If $S=\{x\}$ or $S=\{y\}$, then $\xi_x=2/3$ and $\xi_y=\xi_z=-1/3$. Hence $Q_B=\xi_x\xi_y=-2/9$. If $S=\{z\}$, then $\xi_x=\xi_y=-1/3$ and $\xi_z=2/3$. Therefore $Q_B=\xi_x\xi_y=1/9$. Since each vertex is selected with probability $1/3$, we have $\mathbb E Q_B=-1/9$ and $\mathbb E Q_B^2=1/27$. Then $\operatorname{Var}(Q_B)=2/{81}$.

\item  \textbf{Case 4.} If $e_G(B)=2$, then $G[B]$ is a path on three vertices.  Write its edges as $xy$ and $yz$, so that $y$ is the middle vertex. If $S=\{y\}$, then $\xi_y=2/3$ and $\xi_x=\xi_z=-1/3$. Hence $Q_B=	\xi_x\xi_y+\xi_y\xi_z=-4/9$. If $S=\{x\}$ or $S=\{z\}$, then $\xi_x=2/3$ and $\xi_y=\xi_z=-1/3$. Therefore $Q_B=\xi_x\xi_y+\xi_y\xi_z=-1/9$. Therefore $\mathbb E Q_B=-2/9$ and $\mathbb E Q_B^2=2/27$. Hence $\operatorname{Var}(Q_B)=\mathbb E Q_B^2-(\mathbb E Q_B)^2=2/81$.

Combining the four cases, we obtain
\begin{equation*}
\operatorname{Var}(Q_B)\le\frac19\,e_G(B)=\left(\frac{\ell(3-\ell)}{3(3-1)}\right)^2e_G(B).
\end{equation*}
This proves the assertion for $k=3$. Assume that $k\ge4$. We shall repeatedly use $k-1\le \ell(k-\ell)\le {k^2}/{4}$, which follows from $1\le\ell\le k-1$.
\begin{claim}\label{claimk>41}
If $2k^2-\ell(k-\ell)(k+6)\ge0$, then (\ref{eq:within-linear}) holds.
\end{claim}
\begin{proof}
The coefficient of $\sum_{x\in B}d_{G[B]}(x)^2$ in (\ref{eq:exact-within}) is nonnegative.  Hence we may apply the upper bound from Lemma~\ref{lem:degree-square}.  Substituting it into (\ref{eq:exact-within}) and simplifying gives
\begin{align*}
&\qquad\left(\frac{\ell(k-\ell)}{k(k-1)}\right)^2 e_G(B)-\operatorname{Var}(Q_B)\\&\ge
e_G(B)\left(\frac{\ell(k-\ell)\left(\ell(k-\ell)(k^2+2k-4)-k^2(k-1)\right)}{k^3(k-2)(k-1)^2}+\frac{2\ell(k-\ell)\left(k^2-4\ell(k-\ell)\right)}{k^4(k-2)(k-1)^2}e_G(B)\right).
\end{align*}
By $k^2-4\ell(k-\ell)\ge0$ and $\ell(k-\ell)(k^2+2k-4)-k^2(k-1)\geq(k-1)(2k-4)\geq0$, we obtain (\ref{eq:within-linear}). This proves the claim.
\end{proof}
\begin{claim}\label{claimk>42}
	If $2k^2-\ell(k-\ell)(k+6)<0$ and $0\le e_G(B)\le {k}/{2}$, then~(\ref{eq:within-linear}) holds.
\end{claim}
\begin{proof}
	The coefficient of $\sum_{x\in B}d_{G[B]}(x)^2$ in~(\ref{eq:exact-within}) is negative.  Hence we may apply the lower bound from Lemma~\ref{lem:degree-square}. Substituting this bound into~(\ref{eq:exact-within}) and simplifying, 
	\begin{align*}
\left(\frac{\ell(k-\ell)}{k(k-1)}\right)^2 e_G(B)-\operatorname{Var}(Q_B)\ge e_G(B)f(e_G(B)),
	\end{align*}
	where
	\begin{align*}
	f(e_G(B))&=\frac{\ell(k-\ell)\left(k^4-6k^3+5k^2+\ell(k-\ell)(-2k^2+16k-12)\right)}{k^3(k-3)(k-2)(k-1)^2}\\&\qquad+\frac{2\ell(k-\ell)\left(3k^2(k-1)-\ell(k-\ell)(k^2+10k-12)\right)}{k^4(k-3)(k-2)(k-1)^2}e_G(B).
	\end{align*}
Since $f(e_G(B))$ is linear in $e_G(B)$ and $0\le e_G(B)\le {k}/{2}$, it is enough to verify $f(0)\ge0$ and $f\left({k}/{2}\right)\ge0$. Since $e_G(B)=0$ and $k-1\leq \ell(k-\ell)\leq k^2/4$, we have
\begin{align*}
f(0)&=\frac{\ell(k-\ell)\left(k^4-6k^3+5k^2+\ell(k-\ell)(-2k^2+16k-12)\right)}{k^3(k-3)(k-2)(k-1)^2}\\&\ge\frac{\ell(k-\ell)}{k^3(k-3)(k-2)(k-1)^2}\cdot\min\left\{(k-3)(k-2)^2(k-1),\frac12 k^2(k-2)^2\right\}\ge0.
\end{align*}
Since $e_G(B)=k/2$ and $\ell(k-\ell)\leq k^2/4$, we have
\begin{equation*}
f\left(\frac{k}{2}\right)=\frac{\ell(k-\ell)\left(k^2-k-3\ell(k-\ell)\right)}{k^2(k-3)(k-1)^2}
\ge\frac{\ell(k-\ell)\left(k^2-k-\frac{3k^2}{4}\right)}{k^2(k-3)(k-1)^2}\ge 0.
\end{equation*}
Thus $f(e_G(B))\ge0$ for every $0\le e_G(B)\le {k}/{2}$ and $e_G(B)\ge0$, we have $e_G(B)f(e_G(B))\ge0$. This proves the claim.
\end{proof}
\begin{claim}\label{claimk>43}
	If $2k^2-\ell(k-\ell)(k+6)<0$ and $e_G(B)\ge {k}/{2}$, then~(\ref{eq:within-linear}) holds.
\end{claim}
\begin{proof}
	The coefficient of $\sum_{x\in B}d_{G[B]}(x)^2$ in~(\ref{eq:exact-within}) is negative. Hence we may apply the lower bound from Lemma~\ref{lem:degree-square}. Substituting this bound into~(\ref{eq:exact-within}) and simplifying gives
	\begin{equation*}
\left(\frac{\ell(k-\ell)}{k(k-1)}\right)^2 e_G(B)-\operatorname{Var}(Q_B)\ge e_G(B)\,g(e_G(B)),
	\end{equation*}
where
\begin{equation*}
g(e_G(B))=\frac{\ell(k-\ell)\left(k^3-2k^2+k+\ell(k-\ell)(6-4k)\right)}{k^2(k-3)(k-2)(k-1)^2}+
\frac{2\ell(k-\ell)\left(\ell(k-\ell)-k+1\right)}{k^2(k-3)(k-2)(k-1)^2}e_G(B).
\end{equation*}
	Since $\ell(k-\ell)\ge k-1$, the coefficient of $e_G(B)$ in $g(e_G(B))$ is nonnegative.
	Hence $g(e_G(B))$ is nondecreasing in $e_G(B)$. Therefore, for $e_G(B)\ge{k}/{2}$, it is enough to check $g(k/2)$. We have
	\begin{equation*}
g\left(\frac{k}{2}\right)=\frac{\ell(k-\ell)\left(k^2-k-3\ell(k-\ell)\right)}{k^2(k-3)(k-1)^2}
\ge\frac{\ell(k-\ell)\left(k^2-k-\frac{3k^2}{4}\right)}{k^2(k-3)(k-1)^2}\ge0.
	\end{equation*}
	Thus $g(e_G(B))\ge g\left({k}/{2}\right)\ge0$. Since $e_G(B)\ge0$, we have $e_G(B)\,g(e_G(B))\ge0$. This proves the claim.
\end{proof}
The case $k=3$ together with  Claims~\ref{claimk>41},~\ref{claimk>42} and \ref{claimk>43}, we prove the lemma.
\end{proof}
We now have bounds for both the within-block terms $Q_B$ and the cross-block terms $Q_{BC}$.  Together with the variance decomposition from Lemma~\ref{lem:orthogonality}, these estimates give the desired uniform
bound for $Q_A$.
\begin{theorem}\label{thm:variance}
For every $A\in\binom{[k]}{\ell}$, we have
\begin{equation*}\label{eq:uniform-variance}
\operatorname{Var}(Q_A)\le\left(\frac{\ell(k-\ell)}{k(k-1)}\right)^2 m.
\end{equation*}
\end{theorem}
\begin{proof}[\bf Proof]
By Lemmas \ref{lem:orthogonality},~\ref{lem:cross-variance}~and~\ref{lem:within-linear}, we have
\begin{align*}
\operatorname{Var}(Q_A)&=\sum_{B\in\mathcal B}\operatorname{Var}(Q_B)+\sum_{B<C}\operatorname{Var}(Q_{BC})\\&\le\left(\frac{\ell(k-\ell)}{k(k-1)}\right)^2\left(\sum_{B\in\mathcal B}e_G(B)+\sum_{B<C}e_G(B,C)\right)\\&=\left(\frac{\ell(k-\ell)}{k(k-1)}\right)^2 m.
\end{align*}
Indeed, every edge of $G$ either has both endpoints in one block or joins a unique pair of distinct blocks, and hence is counted exactly once in the two sums above.
\end{proof}

\section{Proof of Main Theorem}\label{sec:Proof of Main Theorem}
In this section, we obtain a simultaneous bound for all $Q_A$ and use it to prove Theorem~\ref{thm:intro-general} and Corollary~\ref{cor:intro-BS}.
\begin{lemma}[\textbf{Cantelli} \cite{Cantelli1928,Savage1961}]\label{lem:cantelli}
Let $Y$ be a random variable with $\mathbb E Y=0$ and $\operatorname{Var}(Y)=\sigma^2$. Then, for every $u>0$,
\begin{equation*}
\mathbb P(Y\ge u)\le\frac{\sigma^2}{\sigma^2+u^2}.
\end{equation*}
\end{lemma}
We now apply Cantelli's inequality to $Q_A-\mathbb E Q_A$ and use the union bound to control all choices of $A$ simultaneously.

\begin{lemma}\label{lem:simultaneous}
There exists a block-coloring outcome such that, for every $A\in\binom{[k]}{\ell}$, we have
\begin{equation*}\label{eq:simultaneous}
Q_A\le\frac{\ell(k-\ell)}{k(k-1)}\sqrt{\left(\binom{k}{\ell}-1\right)m}.
\end{equation*}
\end{lemma}
\begin{proof}[\bf Proof]
	If $m=0$, then $Q_A=0$ for every $A$, and the assertion is immediate. Assume $m>0$. Fix $A\in\binom{[k]}{\ell}$. Since $\mathbb E(Q_A-\mathbb E Q_A)=0$ and, by Theorem~\ref{thm:variance},
	\begin{equation*}
	\operatorname{Var}(Q_A-\mathbb E Q_A)=\operatorname{Var}(Q_A)\le\left(\frac{\ell(k-\ell)}{k(k-1)}\right)^2m,
	\end{equation*}
	Lemma~\ref{lem:cantelli} gives, for every $\varepsilon>0$,
	\begin{equation*}
\mathbb P\left(Q_A-\mathbb E Q_A>\frac{\ell(k-\ell)}{k(k-1)}\sqrt{\left(\binom{k}{\ell}-1+\varepsilon\right)m}\right)\le\frac{1}{\binom{k}{\ell}+\varepsilon}.
	\end{equation*}
Since there are $\binom{k}{\ell}$ choices of $A$, the union bound gives
	\begin{equation*}
\mathbb P\left(\exists A\in\binom{[k]}{\ell}:Q_A-\mathbb E Q_A>\frac{\ell(k-\ell)}{k(k-1)}
\sqrt{\left(\binom{k}{\ell}-1+\varepsilon\right)m}\right)\le\frac{\binom{k}{\ell}}{\binom{k}{\ell}+\varepsilon}<1.
	\end{equation*}
	By Lemma~\ref{lem:block-second}, we have that for every $\varepsilon>0$, there exists a block-coloring outcome such that, simultaneously for all $A\in\binom{[k]}{\ell}$,
	\begin{equation}\label{eq:QA}
Q_A\le\frac{\ell(k-\ell)}{k(k-1)}\sqrt{\left(\binom{k}{\ell}-1+\varepsilon\right)m}.
	\end{equation}
Now take $\varepsilon=1/r$, $r=1,2,\ldots$.  Since the number of
block-coloring outcomes is finite, there exist a fixed outcome
$\omega^\ast$ and a sequence $r_j\to\infty$ such that $\omega^\ast$
satisfies~\eqref{eq:QA} with $r=r_j$ for every $j$.  Thus, for every
$A\in\binom{[k]}{\ell}$,
\begin{equation*}
Q_A\le\frac{\ell(k-\ell)}{k(k-1)}\sqrt{\left(\binom{k}{\ell}-1\right)m}
\end{equation*}
This completes the proof.
\end{proof}
\begin{proof}[\bf Proof of Theorem~\ref{thm:intro-general}]
Choose the block coloring given by Lemma~\ref{lem:simultaneous}. Fix $A\in\binom{[k]}{\ell}$. By Lemmas~\ref{lem:ell-volume}, \ref{lem:decomposition} and \ref{lem:simultaneous}, we obtain
\begin{align*}
e_G\left(\bigcup_{i\in A}V_i\right)&=\frac{\ell}{k}D_A-\frac{\ell^2}{k^2}m+Q_A\\&\le\frac{\ell}{k}\left(\frac{2\ell m}{k}+\frac{\ell(k-\ell)}{k}(n-1)\right)-\frac{\ell^2}{k^2}m+\frac{\ell(k-\ell)}{k(k-1)}\sqrt{
\left(\binom{k}{\ell}-1\right)m}\\&=\frac{\ell^2}{k^2}m+\frac{\ell^2(k-\ell)}{k^2}(n-1)+\frac{\ell(k-\ell)}{k(k-1)}\sqrt{\left(\binom{k}{\ell}-1\right)m}.
\end{align*}
The chosen coloring outcome satisfies this estimate simultaneously
for every $A\in\binom{[k]}{\ell}$.  Therefore
\begin{equation*}
\max_{\substack{A\subseteq[k]\\ |A|=\ell}}e_G\left(\bigcup_{i\in A}V_i\right)\le\frac{\ell^2}{k^2}m+\frac{\ell^2(k-\ell)}{k^2}(n-1)+\frac{\ell(k-\ell)}{k(k-1)}\sqrt{\left(\binom{k}{\ell}-1\right)m}.
\end{equation*}
Finally, all vertices in $V(\widehat G)\setminus V(G)$ are isolated, so deleting them does not change any induced edge count.  By the construction in Section~\ref{sec:Degree Blocks and Coloring}, the resulting partition $V(G)=V_1\cup\cdots\cup V_k$ is balanced. This completes the proof.
\end{proof}
Finally, setting $\ell=2$ in Theorem~\ref{thm:intro-general} and simplifying the remaining terms gives Corollary~\ref{cor:intro-BS}.
\begin{proof}[\bf Proof of Corollary~\ref{cor:intro-BS}]
If $k=2$, then $V_1\cup V_2=V(G)$ and hence $e_G(V_1\cup V_2)=m$, so the assertion holds with equality. Assume $k\ge3$. Applying Theorem~\ref{thm:intro-general} with $\ell=2$, we obtain a balanced partition satisfying
\begin{equation}\label{eq:coro1}
\max_{1\le i<j\le k} e_G(V_i\cup V_j)\le\frac{4m}{k^2}+\frac{4(k-2)}{k^2}(n-1)+\frac{2(k-2)}{k(k-1)}\sqrt{\left(\binom{k}{2}-1\right)m}.
\end{equation}
Moreover,
\begin{equation*}
\left(\frac{2(k-2)}{k(k-1)}\sqrt{\binom{k}{2}-1}\right)^2-\left(\frac{3\sqrt2(k-2)}{k^2}\right)^2
=\frac{2(k-3)(k-2)^2(k^3+2k^2-5k+3)}{k^4(k-1)^2}\ge0,
\end{equation*}
and hence
\begin{equation}\label{eq:coro3}
\frac{2(k-2)}{k(k-1)}\sqrt{\binom{k}{2}-1}\ge\frac{3\sqrt2(k-2)}{k^2}.
\end{equation}
Since $G$ is simple and (\ref{eq:coro3}), we obtain
\begin{align}\label{eq:coro2}
&\qquad\frac{4(k-2)}{k^2}(n-1)+\frac{2(k-2)}{k(k-1)}\sqrt{\binom{k}{2}-1}\,\sqrt m\notag\\&=
\frac{4(k-2)}{k^2}n-\frac{4(k-2)}{k^2}+\frac{3\sqrt2(k-2)}{k^2}\sqrt m+\left(\frac{2(k-2)}{k(k-1)}\sqrt{\binom{k}{2}-1}-\frac{3\sqrt2(k-2)}{k^2}\right)\sqrt m\notag\\&\le
\frac{3(k-2)}{k^2}\sqrt{2m+\frac14}+\left(\frac{4(k-2)}{k^2}+\frac1{\sqrt2}\left(\frac{2(k-2)}{k(k-1)}\sqrt{\binom{k}{2}-1}-\frac{3\sqrt2(k-2)}{k^2}\right)\right)n-\frac{4(k-2)}{k^2}\notag\\&=
\frac{3(k-2)}{k^2}\sqrt{2m+\frac14}+\left(\frac{k-2}{k^2}+\frac{k-2}{k(k-1)}\sqrt{k(k-1)-2}\right)n-\frac{4(k-2)}{k^2}\notag\\&\le\frac{3(k-2)}{k^2}\sqrt{2m+\frac14}+\left(\frac{k-2}{k^2}+\frac{k-2}{k(k-1)}\sqrt{k(k-1)-2}\right)n+\frac{(k-2)(2k-5)}{10}.
\end{align}
Combining (\ref{eq:coro1}) and (\ref{eq:coro2}), this completes the proof.
\end{proof}

\section*{Declaration on the use of AI}
The authors used ChatGPT 5.6 Pro to assist in discussing proof strategies, checking proofs, and improving exposition.

\end{document}